\documentclass[nobib]{tufte-handout}
\usepackage{amsmath,stmaryrd,amssymb,amsthm,url,booktabs,hyperref}

\title{Classifying bi-invariant 2-forms on infinite-dimensional Lie groups%
}
\author{David Michael Roberts\thanks{\href{https://orcid.org/0000-0002-3478-0522}{orcid.org/0000-0002-3478-0522}\\
\url{david.roberts@adelaide.edu.au}\\
\noindent
This document is released under a \href{http://creativecommons.org/publicdomain/zero/1.0/}{CC0 license}.\\
I thank Michael Murray, Alexander Schmeding, John Baez and Thomas Leistner for helpful discussions and suggestions.}}
\date{7 November 2023}

\usepackage{euler}

\usepackage[style=alphabetic,%
            citestyle=alphabetic,%
            useprefix=true,%
            sorting=nyt]{biblatex}
\usepackage{ifpdf}
\ifpdf
  \usepackage[all,pdf]{xy} 
\else
  \input xy 
  \xyoption{all}
  \xyoption{2cell} 
  \xyoption{v2}
\fi

\SelectTips{cm}{11} 

\theoremstyle{definition}
\newtheorem{definition}{Definition}[section]
\newtheorem{remark}[definition]{Remark}
\newtheorem{lemma}[definition]{Lemma}
\newtheorem{theorem}{Theorem}
\newtheorem{proposition}[definition]{Proposition}
\newtheorem{corollary}[definition]{Corollary}
\newtheorem{example}[definition]{Example}

\DeclareMathOperator{\id}{id}
\DeclareMathOperator{\pr}{pr}

\DeclareMathOperator{\Ad}{Ad}

\DeclareMathOperator{\Spin}{Spin}

\def\RR{\mathbb{R}}
\def\ZZ{\mathbb{Z}}

\def\cH{\mathcal{H}}
\def\cK{\mathcal{K}}
\def\PU{\mathsf{PU}}
\def\U{\mathsf{U}}

\def\fracu{\mathfrak{u}}
\def\fracg{\mathfrak{g}}
\def\into{\hookrightarrow}
\def\onto{\twoheadrightarrow}

\begin{document}
\maketitle

\begin{abstract}
A bi-invariant differential 2-form on a Lie group $G$ is a highly constrained object, being determined by purely linear data: an $\Ad$-invariant alternating bilinear form on the Lie algebra of $G$.
On a compact connected Lie group these have an known classification, in terms of de~Rham cohomology, which is here generalised to arbitrary finite-dimensional Lie groups, at the cost of losing the connection to cohomology.
This expanded classification extends further to all Milnor regular infinite-dimensional Lie groups.
I give some examples of (structured) diffeomorphism groups to which the result on bi-invariant forms applies. 
For symplectomorphism and volume-preserving diffeomorphism groups the spaces of bi-invariant 2-forms are finite-dimensional, and related to the de Rham cohomology of the original compact manifold.
In the particular case of the infinite-dimensional projective unitary group $\PU(\cH)$ the classification invalidates an assumption made by Mathai and the author about a certain 2-form on this Banach Lie group.
\end{abstract}

\section{Introduction}

Let $\U=\U(\cH)$ be the unitary group of an infinite-dimensional, separable complex Hilbert space $\cH$, equipped with the norm topology. 
This makes $\U$ a Banach Lie group (eg \cite[\S3]{Schottenloher_18}). The quotient group $\PU = \PU(\cH) = \U(\cH)/U(1)$ is also a Banach Lie group, and $\pi\colon \U\to \PU$ is a principal $U(1)$-bundle and a central extension. 
The characteristic class of this bundle is the generator of $H^2(\PU,\ZZ) = \ZZ$, and further, as $\PU$ is a $K(\ZZ,2)$ and $\U$ is contractible by Kuiper's theorem, this bundle is universal.
A connection on this bundle thus has curvature whose associated cohomology class is (the image in real-valued cohomology of) a generator $c_1$ of $H^2(\PU,\ZZ)$.
Further, this class (the universal first Chern class) is \emph{primitive},\marginnote{Here $\pr_i$ projects \emph{onto} the $i^{th}$ factor, and $m$ is the multiplication map for the group at hand} in the sense that 
\[
    m^*c_1 = \pr_1^*c_1 + \pr_2^*c_1
\]
in $H^2(\PU^2,\ZZ)$.

More generally, one can consider a Lie group $G$ (possibly infinite-dimensional), and a central extension $U(1) \to \widehat{G} \to G$ where the quotient map has smooth local sections.
The first Chern class $c_1(\widehat{G})$ of this circle bundle is primitive in the above sense, and one can consider a connection $A$ whose curvature $F_A$ gives a 2-form representative for (the image in de~Rham cohomology of) $c_1(\widehat{G})$.
The results of \cite{Mur_Stev03} show how the central extension as a Lie group and as a bundle with connection can be reconstructed from the data of the closed ($2\pi i$-)integral 2-form $F_A$ together with a 1-form $\alpha$ on $G\times G$ satisfying 
\begin{equation}\label{eq:Murray-Stevenson_eq}
    m^*F_A + d\alpha = \pr_1^*F_A + \pr_2^*F_A
\end{equation}
and $\delta(\alpha)=0$ on $G^3$, for a specific differential {$\delta\colon \Omega^1(G^2) \to \Omega^1(G^3)$.} 
The equation (\ref{eq:Murray-Stevenson_eq}) then becomes the primitivity condition on passing to de~Rham cohomology.

It is natural to wonder whether it is possible, for a given Lie group $G$ and primitive de~Rham cohomology 2-class on it, to find a 2-form representing that class where the 1-form $\alpha$ vanishes.
In other words, can one lift the primitivity condition to the level of differential forms?

It is a simple result that a primitive differential form $\eta$ is a \emph{bi-invariant} differential form: assuming $\alpha=0$, restricting (\ref{eq:Murray-Stevenson_eq}) to the subspaces $\{g\}\times G$ and $G\times \{g\}$ reduces to $L^*_g\eta = \eta = R^*_g\eta$.
So as a first approximation one can try to understand the existence of bi-invariant 2-forms on a Lie group, before trying to examine its subspace of primitive 2-forms.
The space $\Omega^2_I(G)$ of bi-invariant 2-forms on $G$ is very much smaller than the space $\Omega^2(G)$ of all 2-forms. 
By linearity, the space of primitive differential forms is a subspace of the space of  bi-invariant forms.

In the case of a compact Lie group $G$ the space of bi-invariant forms is already finite-dimensional, and the first result we give here (deduced by assembling known classification results) completely classifies them:

\begin{theorem}
Let $G$ be a connected finite-dimensional Lie group. 
Then there is an isomorphism $\Omega^2_I(G) \xrightarrow{\simeq} \bigwedge^2 \mathfrak{a}^*$, where $\mathfrak{a} = \mathfrak{g}^{ab}$ is the abelianisation of the Lie algebra of $G$. 
If $G$ is compact,\footnote{This result arose in discussion with John Baez around an earlier version of this note.} then $\mathfrak{a}$ is the largest abelian summand of the Lie algebra $\mathfrak{g}$ of $G$.
\end{theorem}

As a corollary, the space of bi-invariant 2-forms vanishes if and only if $[\mathfrak{g},\mathfrak{g}]\subset \mathfrak{g}$ is of codimension $\leq1$, or equivalently $\dim \mathfrak{g}^{ab} \leq 1$.

Regardless of compactness, this focuses our attention on looking at what elements of $\bigwedge^2 \mathfrak{a}^*$ give primitive 2-forms---and there are none!

\begin{corollary}
The space of primitive 2-forms on a finite-dimensional Lie group always vanishes.
\end{corollary}

However, the more interesting case for our purposes is the infinite-dimensional case, and which is of course a source of a number of topologically and group-theoretically nontrivial central extensions as above.

\begin{theorem}
Let $G$ be a connected Milnor regular Lie group, and let $\mathfrak{a} = \mathfrak{g}/\overline{[\mathfrak{g},\mathfrak{g}]}$ be the topological abelianisation of its Lie algebra.
Then there is a vector space isomorphism $\Omega^2_I(G) \simeq \bigwedge^2 \mathfrak{a}^*$.
\end{theorem}

For example, if $G$ any Milnor regular Lie group with topologically perfect Lie algebra (i.e.\ the commutator ideal is dense in the whole algebra), then we get a vanishing result as in the compact semisimple case.

The assumption of Milnor regularity is very generous here, as it holds for all Banach Lie groups and many other infinite-dimensional groups of interest, like loop groups and diffeomorphism groups of (structured) compact manifolds (see \S \ref{sec:inf_dim_case} below).
In particular, it follows from the classification result that both the group $Symp(M,\omega)$ of symplectomorphisms of a compact symplectic manifold $(M,\omega)$, and the group $Diff(M,\mu)$ of volume-preserving diffeomorphisms of a comapct manifold $M$ equipped with a volume form $\mu$, have finite-dimensional spaces of bi-invariant 2-forms.

One can derive further vanishing results from this, using the fact that pullback along a surjective submersion (in the weakest possible sense: surjectivity on tangent spaces) gives an injective map between spaces of differential forms.
This can give rise to examples of infinite-dimensional Lie groups with \emph{non-vanishing} $H^2$ but with no bi-invariant 2-forms.

An example of a Lie group satisfying the hypotheses of this theorem is the Banach Lie group $\PU$ (see Example~\ref{example:none_on_PU} below), which thus admits no non-zero bi-invariant forms.
In \S2.2 of \cite{MR_06} it is assumed there is a connection on $\U\to \PU$ whose curvature 2-form $\omega$ is primitive---which is unfortunate because such a 2-form would be bi-invariant, hence identically zero. 
In fact, we can state a stronger result that is even more damning:

\begin{corollary}
The space of primitive 2-forms on a Milnor regular Lie group always vanishes.
\end{corollary}

The rest of the article is as follows: the next section is dedicated to reviewing some preliminary material and proving the results about finite-dimensional Lie groups, followed by a section covering the main result for infinite-dimensional Lie groups. 
I conclude with a section of examples of infinite-dimensional Lie groups to which these results apply.

\setcounter{theorem}{0}

\section{Bi-invariant forms on finite-dimensional Lie groups}

First, a summary of background results.

To restate the main definitions, recall that a differential $k$-form $\eta$ on a Lie group $G$ (finite or infinite-dimensional) is \emph{bi-invariant} if for all $g\in G$, $\eta$ is invariant under right and left translation by $g$, namely $L^*_g\eta = \eta = R^*_g\eta$. 
The $k$-form $\eta$ is \emph{primitive} if it satisfies
\[
    m^*\eta = \pr_1^*\eta + \pr_2^*\eta
\]
or equivalently $\delta(\eta) = \pr_1^*\eta + \pr_2^*\eta - m^*\eta = 0$.
The subspace of bi-invariant $k$-forms on any Lie group $G$ will be denoted $\Omega^k_I(G)$, and the subspace of primitive forms will be denoted $\Omega^k_p(G)$.
Moreover, $\Omega_I^k$ and $\Omega_p^k$ are contravariant functors and the inclusion is natural, as is the inclusion into the space of all $k$-forms.
Moreover, we have that exterior differentiation restricts to a differential on bi-invariant resp.\ primitive $k$-forms.

\begin{lemma}\label{lemma:surj_subm_pullback_fin_dim}
If $p\colon H\onto G$ be a surjective submersion between Lie groups, then the induced maps $p^*\colon \Omega^k_I(G)\to \Omega^k_I(H)$ and $p^*\colon \Omega^k_p(G)\to \Omega^k_p(H)$ are injective.
\end{lemma}

\begin{proof}
This follows directly from the fact $p^*\colon \Omega^k(G)\to \Omega^k(H)$ is injective, as can be directly calculated: given $\eta$ a $k$-form on $G$ such that $p^*\eta= 0$ and tangent vectors $X_1,\ldots,X_k$ at $g\in G$, choose lifts $\widetilde{X_1},\ldots,\widetilde{X_k}$ based at any $h\in p^{-1}$, and then 
\[
    0 = p^*\eta_h(\widetilde{X_1},\ldots,\widetilde{X_k}) = \eta_g(X_1,\ldots,X_k).
\]
Thus $\eta=0$, and $p^*$ is injective.
\end{proof}

A merely \emph{left} invariant $k$-form is specified completely by its values on the tangent space at the identity, i.e.\ by the linear map $\bigwedge^k\mathfrak{g} \to \RR$.
A \emph{bi-invariant} form $\eta$ is then one that further satisfies
\begin{align*}
(R^*_g\eta)_e(X_1,\ldots,X_k)   & = \eta_g(R_{g*}X_1,\ldots,R_{g*}X_k)\\
                                & = \eta_g(L_{g*}\Ad_{g^{-1}}X_1,\ldots,L_{g*}\Ad_{g^{-1}}X_k)\\
                                & = \eta_e(\Ad_{g^{-1}}X_1,\ldots,\Ad_{g^{-1}}X_k)\\
                                & = \eta_e(X_1,\ldots,X_k)
\end{align*}
and hence the $k$-linear map $\eta_e$ is $\Ad$-invariant.
Thus the restriction map $\Omega^k_I(G) \to \bigwedge^k\mathfrak{g}^*$ factors through the subspace $(\bigwedge^k\mathfrak{g}^*)^{\Ad(G)}$.
Conversely, any element $\eta_0 \in (\bigwedge^k\mathfrak{g}^*)^{\Ad(G)}$ can be used to define a (smooth) bi-invariant $k$-form on $G$, by $\eta_g(gX_1,\ldots,gX_k) = \eta_0(X_1,\ldots,X_k)$.
Thus we have an isomorphism $\Omega_I^k(G) \simeq (\bigwedge^k\mathfrak{g}^*)^{\Ad(G)}$.

It is an old result (see eg \cite[\S4.3]{Reeder_95}) that for a compact connected Lie group $G$ the cochain complex of bi-invariant differential forms doesn't just calculate the de~Rham cohomology of $G$, it is in fact \emph{isomorphic} to the de~Rham cohomology, in that $(\bigwedge^k\mathfrak{g}^*)^{\Ad(G)} \simeq H^k_{dR}(G)$ for each $k$.
This is because the exterior derivative of a bi-invariant form is always $0$.
As a result, for any compact Lie group with vanishing $H^2(G,\RR)$, there are \emph{no} bi-invariant differential forms.

Further the coboundary map $\delta^{CE}$ of the Chevalley--Eilenberg complex of $\fracg$ with values in the trivial $\mathfrak{g}$-module $\RR$, namely
\[
    \delta^{CE}(\eta_e)(X_1,\ldots,X_{k+1}) := \sum_{1\leq i < j \leq k+1} (-1)^{i+j}\eta_e\left([X_i,X_k],X_1,\ldots,\widehat{X}_i,\ldots,\widehat{X}_j,\ldots,X_{k+1}\right)
\]
agrees with that given by the evaluation of $d\eta$ on $T_eG$ (up to a possible pre-factor, depending on conventions).
But since bi-invariant forms are automatically $d$-closed, the resulting alternating $k$-linear map on the Lie algebra is $\delta^{CE}$-closed, and hence is a Lie algebra cocycle.
In the case that $k=2$ this means that $\eta\in (\bigwedge^2\mathfrak{g}^*)^{\Ad(G)}$ satisfies
\begin{equation}\label{eq:2-cocycle_eq}
    \eta([X,Y],Z) + \eta([Y,Z],X) + \eta([Z,X],Y) = 0.
\end{equation}
Further, $\Ad$-invariance means that for all $g\in G$, we have $\eta(\Ad_g X,Y) = \eta(X,\Ad_{g^{-1}}Y)$, and taking $g=\exp(tZ)$ and looking at the $t$-derivative at $0$, we get
\[
    \eta([Z,X],Y) = \eta(X,[-Z,Y]) = -\eta([Y,Z],X).
\]
Thus in fact (\ref{eq:2-cocycle_eq}) reduces to
\begin{equation}\label{eq:reduced_2-cocycle_eq}
    \eta([X,Y],Z) = 0, \quad \forall X,Y,Z\in \mathfrak{g}.
\end{equation}

\begin{lemma}\label{lemma:descending_Ad-invariant_2-cocycles}
The $\Ad$-invariant 2-cocycle $\eta\colon \bigwedge^2 \mathfrak{g}\to \RR$ descends uniquely to an alternating bilinear form $\lambda$ on the abelian Lie algebra $\mathfrak{g}^{ab} = \mathfrak{g}/[\mathfrak{g},\mathfrak{g}]$.
\end{lemma}

\begin{proof}
For $v,w \in \mathfrak{g}^{ab}$, pick $X_v,X_w\in \mathfrak{g}$ lifting them. 
Then define $\lambda(v,w) = \eta(X_v,X_w)$, and one can easily check that given another choice of lift\footnote{Without loss of generality we only need to check one side by skew-symmetry.} $X'_v$ of $v$, hence with $X_v - X'_v$ a sum of commutators, we have 
\[
    \eta(X_v,X_w) = \eta(X'_v + \sum [Y_i,Z_i],X_w) = \eta(X'_v,X_w) + \sum \eta([Y_i,Z_i],X_w) =  \eta(X'_v,X_w),
\]
making $\lambda$ well-defined. Uniqueness is similarly easily checked.
\end{proof}

The last technical result we need is the classification result for compact Lie groups (see eg \cite[Theorem 10.4]{Procesi_07}).

\begin{proposition}
Every compact connected Lie group $G$ is a quotient of a compact Lie group of the form $\widetilde{G} = K_1\times \cdots\times K_n \times U(1)^d$, where the Lie groups $K_j$ are simple and simply connected, by a finite central subgroup $A$ such that $A\cap \left(1\times\cdots\times 1 \times U(1)^d\right) = 1$.
\end{proposition}

Putting all these together we get the first main result.

\begin{theorem}\label{thm:finite_dim_case}
Let $G$ be a connected finite-dimensional Lie group. 
Then there is an isomorphism $\Omega^2_I(G) \xrightarrow{\simeq} \bigwedge^2 \mathfrak{a}^*$, where $\mathfrak{a} = \mathfrak{g}^{ab}$ is the abelianisation of the Lie algebra of $G$. 
If $G$ is compact, then $\mathfrak{a}$ is the largest abelian summand of the Lie algebra $\mathfrak{g}$ of $G$.
\end{theorem}

\begin{proof}
The proof will be done in numbered steps, for ease of referring to them later in the proof of the infinite-dimensional case.

\begin{enumerate}
\item If we have a covering map $\pi\colon \widetilde{G} \to G$ of connected Lie groups, then $\ker\pi < \widetilde{G}$ is a central subgroup and so the adjoint action of $\widetilde{G}$ on its Lie algebra factors through $\widetilde{G} \to G$, and so in fact there is an isomorphism
\[
    (\bigwedge\nolimits^2 \mathfrak{g}^*)^{\Ad(G)} \xrightarrow{\simeq} (\bigwedge\nolimits^2 \mathfrak{g}^*)^{\Ad(\widetilde{G})}.
\]
In particular, this means that $\Omega^2_I(G) \simeq \Omega^2_I(\widetilde{G})$ where $\widetilde{G}$ is the universal covering group of $G$, and so we can consider just the case of a simply-connected Lie group $G=\widetilde{G}$.

\item From Lemma~\ref{lemma:descending_Ad-invariant_2-cocycles} we know that $\eta\in (\bigwedge^2 \mathfrak{g}^*)^{\Ad(G)}$ descends uniquely to an alternating bilinear form on $\mathfrak{a} = \mathfrak{g}/[\mathfrak{g},\mathfrak{g}]$.
That is, $(\bigwedge^2 \mathfrak{g}^*)^{\Ad(G)} \subset \bigwedge^2 \mathfrak{g}^*$ is in the image of the injective map $q^*\colon \bigwedge^2 \mathfrak{a}^* \into \bigwedge^2 \mathfrak{g}^*$ induced from the surjection $q\colon \mathfrak{g}\to \mathfrak{a}$.
As a result there is an injective linear map $(\bigwedge^2 \mathfrak{g}^*)^{\Ad(G)}\to \bigwedge^2 \mathfrak{a}^*$.

\item Using the assumption that $G$ is simply-connected, the Lie algebra homomorphism $q\colon \mathfrak{g}\to \mathfrak{a}$ is the derivative of a unique map of Lie groups $Q\colon G\to \mathfrak{a}$, with codomain the additive group of the abelian Lie algebra $\mathfrak{a}$.
From general properties of Lie groups we know that the homomorphism $Q$ gives us the identity $q(\Ad_g(X)) = \Ad_{Q(g)}(q(X)) = q(X)$, where the latter equality is because $\mathfrak{a}$ is abelian.

We can now check that the image of $\bigwedge^2 \mathfrak{a}^* \into \bigwedge^2 \mathfrak{g}^*$ is inside the $\Ad$-invariant bilinear forms, since if $\lambda \in \bigwedge^2\mathfrak{a}^*$, we have 
\begin{align}
(q^*\lambda)(\Ad_g X,\Ad_g Y) &= \lambda(q(\Ad_g(X)),q(\Ad_g(Y)) \nonumber\\
& = \lambda(q(X),q(Y))\label{eq:ad-invariance_pullback_form}\\
& = (q^*\lambda)(X,Y).\nonumber
\end{align}

\item Thus $(\bigwedge^2 \mathfrak{g}^*)^{\Ad(G)} \simeq(\bigwedge^2 \mathfrak{a}^*)^{G} = \bigwedge^2 \mathfrak{a}^*$, as claimed. 
\end{enumerate}

Finally, in the case that $G$ is compact, it has a finite cover of the form $K_1\times \cdots\times K_n \times U(1)^d$, and hence has Lie algebra a direct sum $\mathfrak{k}_1\oplus \cdots \oplus \mathfrak{k}_n \oplus \mathbb{R}^d$ where each $\mathfrak{k}_j$ is simple and hence $\mathfrak{k}_j^{ab} = \{0\}$.
Thus $\mathfrak{g}^{ab} = \mathbb{R}^d$, the largest abelian direct summand.
\end{proof}

\begin{example}
There are no non-trivial bi-invariant forms not only on the standard examples of compact, simply-connected Lie groups $SU(n)$ ($n\geq 2$) and $\Spin(n)$ ($n\geq 3$) (not to mention the compact forms of the exceptional simple Lie groups), but this is also true for the classical groups $SO(n)$ ($n\geq 2$), $U(n)$ ($n\geq 1$), and $\Spin^c(n)$ ($n\geq 3$).
\end{example}

Looking at the level of the whole Lie group, what this means is that every bi-invariant 2-form $\eta$ on $G$ is the descent of a bi-invariant form on the universal covering group $\widetilde{G}$, and such a form arises by pullback  along the projection $q\colon \widetilde{G} \to \mathfrak{a}$, in the sense that $\pi^*\eta = q^* \lambda$, for $\lambda$ an alternating  bilinear form on $\mathfrak{a}$.

Even better, if the group $\pi_1(G)$ is finite, we know that\marginnote{The pushforward $\pi_*$ sums over the values of the form at all the points in the fibres.} $\frac{1}{|A|}\pi_*\pi^*$ is the identity map on $\Omega^2(G)$, and so given a bi-invariant 2-form $\eta$, it is of the form
\[
    \eta = \frac{1}{|A|}\pi_* \pi^* \lambda = \pi_* q^*\frac{1}{|A|}\lambda.
\]

\begin{remark}\label{rem:actual_symplectic_form}
The 2-form $\lambda$ is an example of a \emph{presymplectic} form on the vector space $\mathfrak{a}$, and one can in fact show that it is the pullback of an actual symplectic form $\omega$ along a surjective linear map $\mathfrak{a} \to V$.
One might as well assume that $(V,\omega)$ is the standard symplectic space $(\mathbb{R}^{2n},\sum_{j=1}^n dx_i\wedge dy_i)$.
Another route to Theorem~\ref{thm:finite_dim_case} is via classifying $\Ad$-invariant \emph{presymplectic} forms on Lie algebras, which arise via pullback of an $\Ad$-invariant symplectic form on a Lie algebra in a similar quotienting process. But $\Ad$-invariant symplectic forms only exist on abelian Lie algebras, as we have here.
\end{remark}

The results for connected Lie groups here can be extended somewhat to the arbitrary finite-dimensional case.

\begin{corollary}\label{cor:non-connected_case_bi-inv_forms}
Let $G$ be a finite-dimensional Lie group with identity component $G_0$. The restriction map $\Omega^2_I(G) \to \Omega^2_I(G_0)$ can be identified with the subspace inclusion $(\bigwedge^2 \mathfrak{a}^*)^{\Ad(\pi_0(G))} \into \bigwedge^2 \mathfrak{a}^*$, where as before, $\mathfrak{a} = \mathfrak{g}^{ab}$.
\end{corollary}

\begin{proof}
This follows from the general result that for a normal subgroup $N \unlhd H$ and a representation of $H$ on $W$, the subspace $W^H$ of $H$-invariants agrees with $(W^N)^{H/N}$, for the induced $H/N$ representation on the $N$-invariants $W^N$.
To prove the corollary, we take $W=\bigwedge^2 \mathfrak{a}^*$, $H = G$ and $N = G_0$, and we know that the action of $G_0$ on $\bigwedge^2 \mathfrak{a}^*$ is trivial, by step 3 in the proof of Theorem~\ref{thm:finite_dim_case}.
\end{proof}

Thus if $G_0$ itself admits no non-zero bi-invariant 2-forms, then neither does $G$.
A concrete example is given by $O(n)$, because $SO(n)$ is covered by the compact simple simply-connected group $\Spin(n)$, and the space of bi-invariant 2-forms vanishes for this latter group.

\begin{remark}
According to \cite[Theorem 4]{Heidenreich_et_al_21}, one can in fact extend the classification of compact connected Lie groups to \emph{arbitrary} compact Lie groups as quotients of certain semidirect products by a finite abelian subgroup that is central in the identity component $G_0$, namely $G\simeq (G_0\rtimes R)/P$. 
It follows from this result that $\pi_0(G)$ is a quotient of $R$, and one might thereby sharpen Corollary~\ref{cor:non-connected_case_bi-inv_forms} if so desired.
\end{remark}

And so we can give a non-existence result for primitive 2-forms.

\begin{corollary}\label{cor:no_primitive_2-forms_fin_dim}
On a finite-dimensional Lie group the space of primitive 2-forms vanishes.
\end{corollary}

\begin{proof}
First consider a \emph{connected} finite-dimensional Lie group $G_0$, which we can take to be simply-connected, since every bi-invariant 2-form is uniquely specified by its value on the Lie algebra $\mathfrak{g}$ and the adjoint action of the universal covering group.

Then the bi-invariant 2-form $\eta$ on $G_0$ will be primitive if and only if the bi-invariant 2-form $\lambda$ on the additive Lie group $\mathfrak{a} =\mathfrak{g}^{ab}$ is, where $\eta$ is the pullback of $\lambda$. 
If we write $\lambda$ in terms of basis elements of $\bigwedge^2 \mathfrak{a}^*$, we can consider a bi-invariant 2-form on $\mathfrak{a}$ associated to generic basis element $e^i\wedge e^j$, and check if it gives rise to a primitive 2-form. 
The 2-form on $\mathfrak{a}$ is the pullback of the standard form $dx\wedge dy$ along some surjection $\mathfrak{a}\to \mathbb{R}^2$.
Thus we can directly calculate
\[
    dx_1 \wedge dy_1 + dx_2 \wedge dy_2 - d(x_1+x_2)\wedge d(y_1+y_2) = - dx_1 \wedge dy_2 - dx_2 \wedge dy_1 \neq 0.
\]
Thus no non-zero bi-invariant 2-form on a $\mathfrak{a}$ is primitive, and hence the same holds for a general connected finite-dimensional Lie group $G_0$.

Lastly, if there were a primitive 2-form on $G$, an arbitrary finite-dimensional Lie group, then it would restrict to a primitive 2-form on $G_0$, hence the space of primitive 2-forms on $G$ must vanish.
\end{proof}

\section{Bi-invariant forms on infinite-dimensional Lie groups}\label{sec:inf_dim_case}

First we need to specify what we mean by an infinite-dimensional Lie group.
The underlying geometric framework is that of \emph{locally convex} smooth manifolds\footnote{For a recent survey see \cite[\S 1]{Schmeding}, see also Appendix E of \emph{op.\ cit.} for the treatment of differential forms in this generality}, which includes all Banach and Fr\'echet manifolds.
Thus ``Lie group'' will always mean a Lie group with underlying manifold being locally convex.

For the purposes of having a workable relationship between Lie groups and Lie algebras, we restrict attention here to \emph{Milnor regular} Lie groups; a general reference for this material is \cite[\S3.3]{Schmeding}. 
We do not need the \emph{definition} of Milnor regularity, but one property that follows from it.
The property we require is that one can define a smooth exponentiation map $\exp\colon \fracg \to G$ and from it, for any Lie algebra element $X\in \fracg$, a smooth curve $\exp(tX)$ in $G$ through the identity element, such that this curve has left logarithmic derivative at $0$ equal to $X$.
Further, the assignment $X\mapsto \exp(tX)$ is smooth in $X \in \fracg$.
This reflects what happens in the finite-dimensional case, and we will work with this property formally without further comment.

By way of examples, every (smooth) Banach Lie group is Milnor regular \cite[Remark 3.33]{Schmeding}, as are examples like: the Fr\'echet Lie group $Diff(M)$, for $M$ a compact manifold \cite[Example 3.36]{Schmeding}; the merely locally convex $Diff(M)$ (modelled on compactly-supported vector fields on $M$) when $M$ is only assumed finite-dimensional, smooth and paracompact \cite[Corollary 13.7]{Glockner_16}; the smooth mapping groups $C^\infty(M,G)$ where $M$ is a compact manifold and $G$ is a Milnor regular Lie group \cite[Proposition 3.49]{Schmeding}.
The direct limit groups $U(\infty) = \bigcup_{n\geq 1} U(n)$ and the analogous $SU(\infty)$ are Milnor regular and locally convex \cite[Theorem 8.1]{Gloeckner_05}.

The definitions of bi-invariant and primitive forms are identical as in the finite-dimensional case, and the results connecting bi-invariant forms to alternating maps on the Lie algebra also hold, only replacing the linear dual by the \emph{continuous} linear dual everywhere. 
Letting $G$ now denote an arbitrary Milnor regular Lie group, we have:

\begin{itemize}
\item The vector space $\Omega_I^k(G)$ of bi-invariant $k$-forms is isomorphic to $(\bigwedge^k \mathfrak{g}^*)^{\Ad(G)}$;
\item Every bi-invariant $k$-form is closed, and defines an $\RR$-valued continuous Lie algebra $k$-cocycle in the Chevalley--Eilenberg complex;
\item Given a surjective homomorphism $p\colon H\to G$ such that $\mathfrak{h}\to \mathfrak{g}$ is surjective\footnote{This is weaker than the usual definition of submersion for locally convex manifolds, \cite[Definition 1.49 (c)]{Schmeding}}, the pullback map $p^*\colon \Omega_I^k(G) \to \Omega_I^k(H)$ is injective.
\item For $\eta \in (\bigwedge^2\mathfrak{g}^*)^{\Ad(G)}$ the equation $\eta([X,Y],Z) = 0$ (that is, equation (\ref{eq:reduced_2-cocycle_eq})) still holds for all $X,Y,Z\in \mathfrak{g}$.
\end{itemize}

The proofs of these results are the same as in the finite-dimensional case, relying on the following technical lemma.

\begin{lemma}[\cite{Schmeding}, Corollary 3.22]\label{lemma:Lie_alg_lctvs}
The Lie algebra of a Lie group (modelled on a locally convex topological vector space) is itself a locally convex space, and the Lie bracket is continuous.
\end{lemma}

As a warmup to a more general statement, we can consider the infinite-dimensional unitary group in the sense of K-theory.
Recall that there is a chain of inclusions of Lie groups
\[
    U(2) \subset U(3) \subset \cdots \subset U(\infty) \subset \U_\mathcal{K} \subset \U
\]
where $U(\infty) := \bigcup_{n\geq 1}U(n)$, and $\U_\cK := \U \cap (1+\cK)$, for $\cK \subset \mathcal{B}(\cH)$ the compact operators on $\cH$. 
The two right-most groups are Banach Lie groups, with Banach Lie algebras $\fracu_\cK \subset \fracu$ respectively. 
Note in particular that $\fracu$ is the Lie algebra associated to the Banach $*$-algebra of skew-adjoint bounded operators on $\cH$, and $\fracu_\cK$ is a 2-sided ideal, the norm-closure of the union $\fracu(\infty) = \bigcup_{n\geq 1} \fracu(n)$ of finite-rank operators.

The result of the previous section is enough to show that if $\omega$ is \emph{any} bi-invariant 2-form on $\U$, then it restricts on $U(\infty)$ to be zero. 
By continuity of $\omega$ on $\U$ and density of $\fracu(\infty)\subset \fracu_\cK$, this then implies that $\omega$ restricts to zero on $\U_\cK$.
However, it doesn't rule out there being a bi-invariant 2-form $\omega'$ on $\U/\U_\cK$ that pulls up to $\U$.
This is especially notable as $\U/\U_\cK$ is a classifying space for $\U_\cK$ (as $\U$ is contractible) and thence for reduced K-theory (eg \cite[\S 7]{BCMMS}). Thus $\pi_2(\U/\U_\cK) \simeq [S^2,\U/\U_\cK] \simeq \widetilde{K}(S^2)\simeq \ZZ$, hence $H^2(\U/\U_\cK,\RR)=\RR$.
We shall see that in contrast to the compact case, this latter cohomology group being non-zero is not sufficient to guarantee the existence of a bi-invariant 2-form.

\begin{remark}
The argument that there is no nonzero bi-invariant 2-form on $U(\infty)$ easily generalises to rule out the existence of nonzero bi-invariant $k$-forms on any Milnor regular Lie group of the form $G_\infty = \bigcup_{n\geq 1}G_n$ with all $G_n$ compact (following \cite{Gloeckner_05}), and such that $H^k(G_n,\RR)=0$ for all $n\gg 1$.
This is because a bi-invariant $k$-form on $G_\infty$ restricts to a bi-invariant $k$-form on $G_n$ for every $n$, and every collection of $k$ tangent vectors at a point in $G_\infty$ live in the tangent space of some $G_n$ for sufficiently large $n$. 
This generalises further to the case of an infinite-dimensional Lie group with merely a dense Lie subgroup of the form $G_\infty$ as above.
\end{remark}

And so we come to the main result about Milnor regular Lie groups.

\begin{theorem}\label{thm:main}
Let $G$ be a connected Milnor regular Lie group, and let $\mathfrak{a} = \mathfrak{g}/\overline{[\mathfrak{g},\mathfrak{g}]}$ be the topological abelianisation of its Lie algebra.
Then there is a vector space isomorphism $\Omega^2_I(G) \simeq \bigwedge^2 \mathfrak{a}^*$.
\end{theorem}

Here $\bigwedge^2 \mathfrak{a}^*$ is the space of continuous alternating bilinear forms on the Mackey-complete locally convex topological vector space $\mathfrak{a}$.
Note that I do not examine here if this isomorphism is a \emph{topological} isomorphism.

\begin{proof}
The proof follows the same strategy as in the finite-dimensional case, except now more care needs to be taken to souce technical theorems in infinite-dimensional Lie theory.
The same numbering is used as in the proof of Theorem~\ref{thm:finite_dim_case}.

\begin{enumerate}
\item The isomorphism $\Omega^2_I(G) \simeq \Omega^2_I(\widetilde{G})$ still holds as $\ker(\widetilde{G}\to G)$ is still central, by connectedness of $\widetilde{G}$.
Thus without loss of generality, we take $G$ simply connected.

\item The equation (\ref{eq:reduced_2-cocycle_eq}) holding means that an $\Ad$-invariant alternating bilinear form $\eta$ on $\mathfrak{g}$ vanishes when either of the arguments is in $[\mathfrak{g},\mathfrak{g}]$, and hence by continuity vanishes when either argument is in the closure of this subspace.
The analogue of Lemma~\ref{lemma:descending_Ad-invariant_2-cocycles} still holds, except now with the \emph{topological} abelianisation $\mathfrak{g}/\overline{[\mathfrak{g},\mathfrak{g}]}$.
There is again a linear injective map $(\bigwedge^2\mathfrak{g}^*)^{\Ad(G)}\to \bigwedge^2\mathfrak{a}^*$.

\item Now here we need to rely on the fact that the Lie algebra of a Milnor regular group is Mackey complete \cite[Remark II.5.3.(b)]{Neeb_06} (see also \cite[Theorem 4, 2)]{Hanusch_22}), and locally convex by Lemma \ref{lemma:Lie_alg_lctvs}, and that the additive group of a Mackey complete lctvs is a Milnor regular group \cite[Example 3.34]{Schmeding}. 
And, further, Lie's second theorem holds for Milnor regular Lie groups \cite[Proposition E.14]{Schmeding}, so that the Lie algebra map $q\colon \mathfrak{g}\to \mathfrak{a}$ integrates (as $G$ is assumed simpy-connected) to a unique Lie group homomorphism $Q\colon G\to \mathfrak{a}$.
Further, the identity (\ref{eq:ad-invariance_pullback_form}) still holds (using eg \cite[Exercise 3.2.11(b)]{Schmeding} applied to $Q$), and so the pullback of a bilinear form $\lambda \in \bigwedge^2\mathfrak{a}^*$ is $\Ad$-invariant.

\item We conclude as in the finite-dimensional case. \qedhere

\end{enumerate}

\end{proof}

We can give a sufficient condition for vanishing of the space of bi-invariant 2-forms in terms of continuous Lie algebra cohomology and the \emph{second} term of the lower central series.

\begin{proposition}\label{prop:main}
Let $G$ be a Milnor regular Lie group whose Lie algebra $\mathfrak{g}$ satisfies
\begin{enumerate}
    \item $H^2_c(\fracg,\RR) = 0$, and 
    \item $[\fracg,[\fracg,\fracg]] \subset [\fracg,\fracg]$ is dense.
\end{enumerate}
Then the space of bi-invariant differential 2-forms vanishes.
\end{proposition}

\begin{proof}
Assume to start just that the first condition holds.
Recall that a bi-invariant differential form on $G$ gives a Lie algebra cocycle, and so if the Lie algebra cohomology group vanishes, every 2-cocycle $\eta$ on $\fracg$ is a coboundary, hence of the form $\eta(X,Y) = p([X,Y])$ for some continuous linear functional $p\colon \fracg\to \RR$.

However, from equation (\ref{eq:reduced_2-cocycle_eq}), we have $p([Z,[X,Y]]) =-p([[X,Y],Z]) = -\eta([X,Y],Z) = 0$, for all $X,Y,Z\in \fracg$.
Continuity then implies that $p$ vanishes on $\overline{[\fracg,[\fracg,\fracg]]}$, where this closure is taken inside $[\fracg,\fracg]$.
Thus $p$ arises in fact from a continuous functional on $[\fracg,\fracg]/\overline{[\fracg,[\fracg,\fracg]]}$.
Now, taking also the second assumption the space $\big([\fracg,\fracg]/\overline{[\fracg,[\fracg,\fracg]]}\big)^*$ is trivial, so only possible bi-invariant differential form is $0$.
\end{proof}

We have the following simple characterisation, which holds in practice in several examples below.

\begin{corollary}\label{cor:first}
If $G$ is a connected Milnor regular Lie group whose Lie algebra is topologically perfect then are no non-zero bi-invariant 2-forms on $G$.
\end{corollary}

The proof is immediate from Theorem~\ref{thm:main}, because it means $\overline{[\mathfrak{g},\mathfrak{g}]} = \mathfrak{g}$.
For instance, if $\fracg$ is non-abelian and topologically simple (there are no nonzero proper closed ideals), the vanishing result is immediate.

And we also put the final nail in the coffin in the dream of a primitive curvature 2-form.

\begin{corollary}\label{cor:no_primitive_2-forms}
The space of primitive 2-forms on a Milnor regular Lie group $G$ always vanishes.
\end{corollary}

\begin{proof}
As in the finite-dimensional case, we can reduce to a connected and simply-connected Lie group.
Further, we only need consider the abelian case, since that is where all the bi-invariant 2-forms arise.
And then, assuming the topological abelianisation $\mathfrak{a}$ of $\mathfrak{g}$ is at least 2-dimensional, we can find some $\eta\in \bigwedge^2 \mathfrak{a}^*$ and a 2-dimensional subspace $V\subset \mathfrak{a}$ with $\eta$ non-zero on it.
A 2-form $\eta$ is primitive if $\delta(\eta) = \pr_1^*\eta + \pr_2^*\eta - a^*\eta = 0$, for $a$ the addition map.
We can consider the commutative square
\[
    \xymatrix{
        \Omega^2_I(\mathfrak{a}) \ar[r]^\delta \ar[d] & \Omega^2_I(\mathfrak{a}^2) \ar[d]\\
        \Omega^2_I(V) \ar[r]_\delta & \Omega^2_I(V^2)
    }
\]
where the vertical, surjective maps are restrictions.
Then the same calculation as in Corollary~\ref{cor:no_primitive_2-forms_fin_dim}, performed with basis vectors of $V$, shows that $\delta(\eta)$ cannot vanish in $\Omega^2_I(\mathfrak{a}^2)$, as the image in $\Omega^2_I(V^2)$ in non-zero.
Thus no bi-invariant 2-form is primitive.
\end{proof}

\section{Infinite-dimensional examples}\label{sec:inf_dim_examples}

From the previous corollary, we can immediately conclude the following:

\begin{example}\label{example:none_on_PU}
There are no nonzero bi-invariant 2-forms on $\U$, $\PU$ or $\U/\U_\cK$.

Let us first show it for $\U$, using the fact that $\fracu=[\fracu,\fracu]$ as noted by Ayupov \cite[Remark, page 179]{Ayupov_96} (following \cite{Halmos_54}) that every skew-adjoint operator is the sum of four commutators of skew-adjoint operators.
Then a bi-invariant 2-form on $\PU$ or $\U/\U_\cK$ pulls back to give a bi-invariant 2-form on $\U$, which necessarily vanishes.
For $\U/\U_\cK$ we can even verify the result directly, since $\fracu/\fracu_\cK$ is simple. 
\end{example}

As noted above, this means we see that non-zero $H^2(\U/\U_\cK,\RR)$ is now insufficient to ensure the existence of bi-invariant 2-forms.

\begin{example}
For a compact smooth boundaryless manifold $M$, the Lie algebra $Vect(M)$ of smooth vector fields is the Lie algebra of the Milnor regular group $Diff(M)$. 
Further, $Vect(M)$ is simple (eg \cite[Theorem 1.4.3]{Banyaga}), and so Corollary \ref{cor:first} applies.
\end{example}

The Lie algebra $Vect_c(M)$ of \emph{compactly-supported} smooth vector fields on a non-compact smooth finite-dimensional manifold $M$ is also perfect \cite[Theorem 1.4.3]{Banyaga}. 
As the diffeomorphism group $Diff(M)$ is regular in this case too, and its Lie algebra is $Vect_c(M)$, one can use the proposition for full diffeomorphism groups of finite-dimensional (paracompact) manifolds as well.

\begin{example}\label{example:symp}
Let $(M,\omega)$ be a compact symplectic manifold. Then the Lie algebra $Vect_{Ham}(M,\omega)$ of Hamitonian vector fields is perfect, and this is the commutator ideal of the Lie algebra $Vect(M,\omega)$ of all symplectic vector fields.
Further, these two Lie algebras correspond to the regular Lie groups $Ham(M,\omega)$ and $Symp(M,\omega)$ respectively (Hamiltonian symplectomorphisms, and all symplectomorphisms, see eg \cite[Remark V.2.14(d) and Theorem III.3.1]{Neeb_06}), and so condition 2 of Proposition \ref{prop:main} holds for $Symp(M,\omega)$. 
Moreover Corollary~\ref{cor:first} applies outright to $Ham(M,\omega)$, which thus has no non-zero bi-invariant 2-forms.

In fact we can say more: there is an isomorphism $H^1_{dR}(M)\xrightarrow{\simeq} H^2_c(Ham(M,\omega),\RR)$ \cite[Theorem 2.3]{Janssens-Vizman_16}, and there is a short exact sequence of Lie algebras
\[
    0\to Vect_{Ham}(M,\omega) \to Vect(M,\omega) \to H^1_{dR}(M) \to 0
\]
where the quotient is considered as an abelian Lie algebra. 
In the context of Theorem \ref{thm:main} we have that the topological abelianisation of $Vect(M,\omega)$ is $H^1_{dR}(M)$, and so $\Omega^2_I(Symp(M,\omega)) \simeq \bigwedge^2 H^1_{dR}(M)^*$, which is finite-dimensional.
\end{example}

\begin{example}\label{example:vol}
The Lie algebra $Vect(M,\mu)$ of divergence-free vector fields on a compact manifold $M$ equipped with a volume form $\mu$ has perfect commutator ideal \cite{Lichnerowicz_74}, and is the Lie algebra of the Milnor regular Lie group $Diff(M,\mu)$ of volume-preserving diffeomorphisms (see eg \cite[Theorem III.3.1]{Neeb_06} for the regularity result). Further, the commutator ideal $[Vect(M,\mu),Vect(M,\mu)]$ is the subalgebra $Vect_{ex}(M,\mu)$ of \emph{exact} divergence-free vector fields.
A result of Roger \cite{Roger_95} (see \cite[Theorem 2.1]{Janssens-Vizman_16}) states that there is an isomorphism $H^2_{dR}(M) \xrightarrow{\simeq} H^2_c(Vect_{ex}(M,\mu),\RR)$. 

Since the exact divergence-free vector fields are the Lie algebra of the Fr\'echet Lie group $Diff_{ex}(M,\mu)$ of exact volume-preserving diffeomorphisms, and this group is Milnor regular (\cite[Theorem 8.5.2]{Omori_74}, using the fact ILH Lie groups are Milnor regular \cite[\S6]{KMOY_82}) then Proposition \ref{prop:main} applies. 
Thus there are no non-zero bi-invariant 2-forms on $Diff_{ex}(M,\mu)$.
If $M$ is $n$-dimensional there is a short exact sequence of Lie algebras
\[
    0 \to Vect_{ex}(M,\mu) \to Vect(M,\mu) \to H^{n-1}_{dR}(M)\to 0,
\]
making $H^{n-1}_{dR}(M)$ the topological abelianisation of $Vect(M,\mu)$, and hence that $\Omega^2_I(Diff(M,\mu)) \simeq \bigwedge^2 H^{n-1}_{dR}(M)^*$, which is again finite-dimensional.

Since the geometry of $Diff(M,\mu)$ is related to the study of the flow of an incompressible fluid on $M$, it might be interesting to know what is the relation between the existence of a bi-invariant 2-form on $Diff(M,\mu)$  and the fluid dynamics.
\end{example}

As a further result, consider the construction of a central extension $\widehat{G}\to G$ of (Fr\'echet) Lie groups\footnote{It is assumed that $G$ is simply-connected, otherwise one needs a differential character on it, not just a closed 2-form} given in \cite[\S 3]{Mur_Stev03}, from differential form data $(\alpha,R)$. 
Here $R$ a closed, $2\pi i$-integral 2-form on $G$, $\alpha$ a 1-form on $G^2$, and they satisfy
\begin{align*}
 \delta(R) &:= \pr_1^*R+\pr_2^*R - m^*R = d\alpha\\
\delta(\alpha)&:= \pr_{23}^*\alpha - (m\times \id)^*\alpha + (\id\times m)^*\alpha - \pr_{12}^*\alpha = 0
 \end{align*}
The 1-form $\alpha$ measures the failure of a connection 1-form $A$ with curvature $R$ to be primitive: the pullback of $\alpha$ to $\widehat{G}^2$ is $\delta(A)$.

Example \ref{example:none_on_PU} implies that in order to reconstruct $\U\to \PU$ (or indeed any nontrivial central extension of $\PU$) from data relating to $\PU$ alone, one must find a nonzero 1-form $\alpha$ on $\PU^2$ with $d\alpha = \delta(R)$, where $R$ is any 2-form on $\PU$ representing $c_1(\U)$. 
More generally, Corollary~\ref{cor:no_primitive_2-forms} shows there is always a nontrivial obstruction to the flatness of $\alpha$ for central extensions of Milnor regular Lie groups, for instance the well-known central extension of $Ham(M,\omega)$ by the quantomorphism group \cite{Kostant}, and the Ismagilov central extensions of $Diff_{ex}(M,\mu)$ associated to integral 2-forms on $M$ \cite[\S25.3]{Ismagilov_96}.

\printbibliography
\end{document}